\documentclass[makeidx]{amsart}

\usepackage{a4wide}
\usepackage{latexsym}
\usepackage{amsmath}
\usepackage{amssymb}
\usepackage{amsthm}
\usepackage{amsbsy}
\usepackage[dvips]{graphicx}

\newcommand{\e}{\varepsilon}
\newcommand{\LEE}{\mathcal{L}}
\newcommand{\R}{\mathbb{R}}
\newcommand{\F}{\mathcal{F}}
\newcommand{\Mid}{\,\, \Big| \,\,}

\theoremstyle{plain}
\newtheorem*{thm}{Theorem}
\newtheorem*{prop}{Proposition}
\newtheorem*{lem}{Lemma}
\newtheorem*{TB}{Thurston-Bennequin inequality}
\newtheorem*{TBC}{Thurston-Bennequin inequality for convex hypersurfaces}

\title[Thurston-Bennequin inequality for a hypersurface]
{
On the violation of Thurston-Bennequin inequality 
for a certain non-convex hypersurface
}
\author[A.~Mori]
{
Atsuhide MORI
}

\address{Graduate~School~of~Science, Osaka~University, 
Toyonaka, Osaka~560-0043, Japan}

\email{ka-mori@ares.eonet.ne.jp}

\subjclass[2000]{Primary~57R17, 57R20, secondary~57R30}

\keywords{Contact structure, characteristic class, convex contact geometry}

\begin{document}

\begin{abstract}
We show that any open subset of a contact manifold of dimension greater than three 
contains a certain hypersurface $\Sigma$ 
which violates the Thurston-Bennequin inequality. 
We also show that no convex hypersurface smoothly approximates $\Sigma$.
These results contrasts with the $3$-dimensional case, where 
any surface in a small ball satisfies 
the inequality (Bennequin\cite{Bennequin}) 
and is smoothly approximated by a convex one (Giroux\cite{GirouxConvex}).
\end{abstract}
\maketitle
\section{Introduction and preliminaries}
Tightness is a fundamental notion
in $3$-dimensional contact topology.
It is characterized by the Thurston-Bennequin inequality
(see \S 1.1). Roughly, for a Seifert (hyper)surface 
$\Sigma$ in a contact manifold, this inequality 
compares the contact structure along $\Sigma$
with the tangent bundle of $\Sigma$ by means of 
relative euler number.
An overtwisted disk is a $2$-disk equipped
with a certain germ of $3$-dimensional contact structure, for which the inequality fails.
A contact manifold is said to be tight if it contains no embedded overtwisted disks.
Then the inequality automatically holds for any $\Sigma$ (Eliashberg\cite{Eliashberg3}).
The $1$-jet space $J^1(\R,\R)$ for a function $f:\R\to\R$ is tight (Bennequin\cite{Bennequin}).
Thus all contact $3$-manifolds, which are modelled on $J^1(\R,\R)$, are locally tight.

Giroux\cite{GirouxConvex} smoothly approximated a given 
compact surface in a contact $3$-manifold by
a surface with certain transverse monotonicity,
i.e., a convex surface (see \S1.2 for the precise definition). 
Thus, in $3$-dimensional case, we may easily consider that the above inequality is only for convex surfaces. 
Contrastingly, in higher dimension, we show the following theorem.
\begin{thm}
In the case where $n>1$, any open subset of 
the $1$-jet space $J^1(\R^{n},\R)$,
which is the model space for contact $(2n+1)$-manifolds,
contains a hypersurface $\Sigma$ such that
\begin{enumerate}
\item $\Sigma$ violates the Thurston-Bennequin inequality, \quad and
\item no convex hypersurface smoothly approximates $\Sigma$.
\end{enumerate}
\end{thm}
\noindent
This leads us to seriously restrict the inequality
to convex hypersurfaces (see \cite{Mori1} for a sequel).

\subsection{Thurston-Bennequin inequality}
Let $\Sigma$ be a compact oriented hypersurface
embedded in a positive contact manifold $(M^{2n+1},\alpha)$ 
$(\alpha\wedge (d\alpha)^n>0)$ which tangents to
the contact structure $\ker\alpha$ 
at finite number of interior points.
Let $S_+(\Sigma)$ (resp. $S_-(\Sigma)$) denote the set of positive
(resp. negative) tangent points.
With respect to the symplectic structure $d\alpha|\ker\alpha$,
the symplectic orthogonal of $T\Sigma\cap \ker\alpha$ defines 
a singular line field $L\subset T\Sigma$.
The integral foliation $\F_{\Sigma}$ of $L$ on $\Sigma$ is 
called the {\em characteristic foliation}.
The singularity of $L$ coincides with $S_+(\Sigma)\cup S_-(\Sigma)$.
The restriction $\gamma=\alpha|T\Sigma$ defines a 
holonomy invariant transverse contact structure
of $\F_\Sigma$ and determines the orientation of $L$
(i.e., $X\in L_{>0}\, \Longleftrightarrow\, \iota_X(^\exists dvol_\Sigma)=\gamma\wedge (d\gamma)^n
\, \Longrightarrow\, \gamma\wedge\LEE_X\gamma=0$).
We define the index $\mathrm{Ind} \, p=\mathrm{Ind}_L \, p$ 
of each tangent point $p\in S_\pm(\Sigma)$ by regarding it 
as a singular point of $L$. 
Assume that the boundary of each connected component of $\Sigma$ is
non-empty and $L$ is outwards transverse to $\partial\Sigma$. 
Then the boundary $\partial \Sigma$ is said to be {\em contact-type}.

The unit $2$-disk $D^2$ equipped with 
the germ of contact structure $\ker\{(2r^2-1)dz+r^2(r^2-1)d\theta\}$ 
is called an (the) overtwisted disk, where $(r,\theta,z)$ is the 
cylindrical coordinates of $D^2\times \R$. 
Slightly extending $D^2$, we obtain a disk with contact-type boundary
such that the singularity of the characteristic foliation is a single
negative sink point. A contact $3$-manifold is said to be
{\em overtwisted} or {\em tight} depending on whether it contains
an embedded overtwisted disk (with the same germ as above) or not.
Let $\Sigma$ be {\em any} surface with contact-type boundary
embedded in the $1$-jet space $J^1(\R,\R)(\approx \R^3)$
equipped with the canonical contact form.
Then Bennequin\cite{Bennequin} proved the following inequality
which immediately implies the tightness of $J^1(\R,\R)$:

\begin{TB}
$
\displaystyle \sum_{p\in S_-(\Sigma)}\mathrm{Ind}\,p \leq 0.
$
\end{TB}

\noindent
Eliashberg proved the same inequality for symplectically fillable
contact $3$-manifolds (\cite{Eliashberg2}),
and finally for all tight contact $3$-manifolds (\cite{Eliashberg3}).
The inequality can be written in terms of
relative euler number. Let $X$ be the above vector field on 
a hypersurface $\Sigma\subset (M^{2n+1},\alpha)$ with contact-type boundary.
Then, since $X\in T\Sigma \cap \ker\alpha$, we can regard $X$
as a section of $\ker\alpha|\Sigma$ which is canonical
near the boundary $\partial\Sigma$. Thus we can define the relative
eular number of $\ker\alpha|\Sigma$ by
$$
\langle e(\ker\alpha),\,[\Sigma,\partial\Sigma]\rangle
=\sum_{p\in S_+(\Sigma)}\mathrm{Ind}\,p-\sum_{p\in S_-(\Sigma)}\mathrm{Ind}\,p.
$$ 
Then the Thurston-Bennequin inequality may be expressed as 
$$
-\langle e(\ker\alpha),\,[\Sigma,\partial\Sigma]\rangle \leq -\chi(\Sigma).
$$
There is also an absolute version of the Thurston-Bennequin inequality
which is expressed as
$|\langle e(\ker\alpha),\,[\Sigma]\rangle|\leq -\chi(\Sigma)$,
or equivalently 
$$
\sum_{p\in S_-(\Sigma)}\mathrm{Ind}\,p \leq 0\quad\quad
\mathrm{and}\quad\quad
\sum_{p\in S_+(\Sigma)}\mathrm{Ind}\,p \leq 0
$$
for any closed hypersurface $\Sigma$ with $\chi(\Sigma)\leq 0$.
This holds if the euler class $e(\ker \alpha)$ is a torsion, 
especially if $H^{2n}(M;\R)=0$. 
Note that the inequality and its absolute version can be defined
for any oriented plane field on $M^3$
(see Eliashberg-Thurston\cite{EliashbergThurston}).
They are originally proved for codimension $1$ foliations on $M^3$ without Reeb components
by Thurston (see \cite{Thurston}). 

\subsection{Convex hypersurfaces}
A vector field $X$ on $(M^{2n+1},\alpha)$ is said to be {\em contact}
if the Lie derivative $\LEE_X\alpha$ vanishes on $\ker\alpha$.
Let $V$ denote the space of contact vector fields 
on $(M^{2n+1},\alpha)$. We can see that the linear map 
$\alpha(\cdot):V \to C^\infty(M^{2n+1})$
is an isomorphism. The function $\alpha(X)$ is
called the contact Hamiltonian function of $X$. 
A closed oriented hypersurface $\Sigma \subset (M^{2n+1},\alpha)$ is
said to be {\it convex} if there exists a contact vector field
$Y$ on a neighbourhood $\Sigma\times(-\e,\e)$ 
of $\Sigma=\Sigma\times\{0\}$ with 
$Y=\partial/\partial z$ ($z\in(-\e,\e)$), i.e., 
$Y$ is positively transverse to $\Sigma$ (Giroux\cite{GirouxConvex}).
By perturbing the contact Hamiltonian function if necessary, 
we may assume that $\Gamma=\{\alpha(Y)=0\}$ is a hypersurface 
transverse to $\Sigma$. Then $\Gamma$ separates
$\Sigma$ into the {\em positive region}
$\Sigma_+=\{\alpha(Y)\ge 0\}$ and the {\em negative region}
$\Sigma_-=-\{\alpha(Y)\le 0\}$ so that
$\Sigma=\Sigma_+\cup (-\Sigma_-)$.
Each interior $\mathrm{int}\,\Sigma_\pm$
has the positive exact symplectic form $\displaystyle d\left(\frac{\pm 1}
{\alpha(Y)}\alpha\right)\Big|\mathrm{int}\,\Sigma_\pm$.
We can modify the function
$\displaystyle \left| \frac{1}{\alpha(Y)} \right|: \Sigma\times(-\e,\e)
\to \R_{>0}\cup \{\infty\}$ 
near $\Gamma$ to obtain a function 
$f: \Sigma\times(-\e,\e)\to \R_{>0}$
such that $d(f\alpha)|\mathrm{int}\,\Sigma_\pm$ are 
symplectic and 
$f\alpha$ is $Y$-invariant. 
(This is the ``transverse monotonicity'' of $\Sigma$.)
Note that the dividing set $\Gamma\cap\Sigma$ is then the convex ends
of the exact symplectic manifolds $\mathrm{int}\,\Sigma_\pm$.

A {\em convex hypersurface with contact-type boundary} 
is a connected hypersurface $\Sigma$ which admits 
a transverse contact vector field $X$ such that, 
for the associated decomposition $\Sigma=\Sigma_+\cup (-\Sigma_-)$,
$\mathrm{int}\,\Sigma_\pm$ are also
convex exact symplectic manifolds, and the contact-type boundary 
$\partial\Sigma=\partial \Sigma_+ \setminus \partial \Sigma_-$ is non-empty.
(Changing $X$ if necessary, we can assume moreover that 
the dividing set $\Gamma\cap\Sigma$ contains $\partial\Sigma$.) 
Then the Thurston-Bennequin inequality 
can be expressed as follows.

\begin{TBC}
$
\displaystyle \chi(\Sigma_-)\leq 0.
$
\end{TBC}

Slightly extending the overtwisted disk $D^2$, we obtain a
convex disk $\Sigma$ which is the union $\Sigma_+\cup(-\Sigma_-)$ of 
a disk $\Sigma_-$ and an annulus $\Sigma_+$ surrounding $\Sigma_-$.
Then $\Sigma$ violates the Thurston-Bennequin inequality ($\chi(\Sigma_-)=1>0$),
and is called a convex overtwisted disk.
A possible higher dimensional overtwisted convex hypersurface
$\Sigma$ would also satisfy $\chi(\Sigma_-)>0$
and $\partial\Sigma_+\setminus\partial\Sigma_-\neq\emptyset$.
Particularly $\partial\Sigma_+$ would have to be disconnected (see \cite{Mori1}).

\section{Proof of Theorem}
We show the following Proposition. 
\begin{prop}
Let $(M^3,\alpha)$ be an overtwisted contact $3$-manifold 
and $B_\e^{2n-2}$ the $\e$-ball in $\R^{2n-2}$ $(0<\e\ll 1)$. 
Then there exist a closed hypersurface $\widetilde{\Sigma}$
and a hypersurface $\Sigma$ with contact-type boundary 
in the product contact $(2n+1)$-manifold
$$
\left(M^3\times B_\e^{2n-2}\ni(p,(x_1,y_1,\dots, x_{n-1},y_{n-1})),\,
\beta=\pi^*\alpha+\sum_{i=1}^{n-1}\left(x_idy_i-y_idx_i\right)\right)
$$
such that $\widetilde{\Sigma}$ and $\Sigma$ are not convex, 
$\Sigma\subset\widetilde{\Sigma}$, 
and $\Sigma$ violates the Thurston-Bennequin inequality, 
where $\pi$ denotes the natural projection to $M^3$.
\end{prop}
\begin{proof}
Let $(r,\theta,z)$ be the cylindrical coordinates of $\R^3$, and consider the functions
$$
\lambda(r)=2r^2-1 \quad 
\mathrm{and} \quad
\mu(r)=r^2(r^2-1).
$$
Then we see that the contact structure on $\R^3$ defined by the contact form
$$
\alpha'=\lambda(r)dz+\mu(r)d\theta
$$
is overtwisted. 
An overtwisted disk in $(M^3,\alpha)$ has a neighbourhood 
which is contactomorphic to $U=\{\e^{-2}z^2+r^2<1+2\e\}\,\subset (\R^3,\alpha')$. 
Thus, by using the formula
$$
f^2\sum_{i=1}^{n-1}(x_idy_i-y_idx_i)=\sum_{i=1}^{n-1}(fx_id(fy_i)-fy_id(fx_i)) 
\quad 
\left(\forall f \in C^\infty\left(M^3\times\R^{2n-1}\right)\right),
$$
we can replace $(M^3, \alpha)$ in Proposition with $(U, \alpha'|U)$. 
Then we take the hypersurface
$$
\widetilde{\Sigma}=\left\{
(z,r,\theta,x_1,y_1,\dots,x_{n-1},y_{n-1})
\Mid
r^2+\e^{-2}\left(z^2+\sum_{i=1}^{n-1}(x_i^2+y_i^2)\right)=1+\e
\right\}
$$ 
and its subset
$$
\Sigma=\left\{
(z,r,\theta,x_1,y_1,\dots,x_{n-1},y_{n-1})\in \widetilde{\Sigma}
\Mid
r-z\le 1
\right\}.
$$ 
We orient $\widetilde{\Sigma}$ so that the characteristic foliation $\F_{\widetilde{\Sigma}}$ 
is presented by the vector field
$$\begin{array}{rl}
X\hspace{-3mm}
&= \displaystyle
\e^{-2}r(r^2-1)z\partial_r
+(1+2\e-2\e^{-2}z^2)\partial_\theta
+\left\{
(r^2-1)^2+(2r^2-1)(\e^{-2}z^2-\e)
\right\}\partial_z
\\ & \displaystyle
\quad +\,\e^{-2}(2r^2-1)z\sum_{i=1}^{n-1}
\left(
x_i\partial_{x_i}+y_i\partial_{y_i}
\right)
+\e^{-2}(2r^4-2r^2+1)\sum_{i=1}^{n-1}
\left(
-y_i\partial_{x_i}+x_i\partial_{y_i}
\right).
\end{array}
$$
Indeed the following calculations shows that 
the vector field $X$ satisfies $X \in T\widetilde{\Sigma}$,
$(\beta|T\widetilde{\Sigma})(X)=0$, and 
$\mathcal{L}_X (\beta|T\widetilde{\Sigma})
=2\e^{-2}(2r^2-1)z\beta|T\widetilde{\Sigma}$.
$$
\begin{array}{l}
\displaystyle
\left\{
2rdr+\e^{-2}\left(2zdz+2\sum_{i=1}^{n-1}(x_idx_i+y_idy_i)\right)
\right\}
(X)
\\
\displaystyle
= 2\e^{-2}(2r^2-1)z
\left\{
r^2+\e^{-2}
\left(
z^2+\sum_{i=1}^{n-1}(x_i^2+y_i^2)
\right)-1-\e
\right\},
\end{array}
$$
$$
\beta=(2r^2-1)dz+r^2(r^2-1)d\theta+\sum_{i=1}^{n-1}\left(x_idy_i-y_idx_i\right),
$$
$$
\begin{array}{rl}
\beta(X)
& \displaystyle
=(2r^2-1)\left\{
(r^2-1)^2+(2r^2-1)(\e^{-2}z^2-\e)
\right\}
\\ 
& \displaystyle
\quad +r^2(r^2-1)\{1-2(\e^{-2}z^2-\e)\}+\e^{-2}(2r^4-2r^2+1)\sum_{i=1}^{n-1}(x_i^2+y_i^2)
\\ 
& \displaystyle
=(2r^4-2r^2+1)
\left\{
r^2+\e^{-2}
\left(
z^2+\sum_{i=1}^{n-1}(x_i^2+y_i^2)
\right)-1-\e
\right\},
\end{array}
$$
$$
d\beta=4rdr\wedge dz+2r(2r^2-1)dr\wedge d\theta+2\sum_{i=1}^{n-1}dx_i\wedge dy_i,
$$
$$
\begin{array}{rl}
\iota_Xd\beta
& =\e^{-2}r(r^2-1)z(4rdz+2r(2r^2-1)d\theta)-2r(2r^4-2r^2+1)dr
\\
& \displaystyle
\quad +2\e^{-2}(2r^2-1)z\sum_{i=1}^{n-1}(x_idy_i-y_idx_i)
-2\e^{-2}(2r^4-2r^2+1)\sum_{i=1}^{n-1}(x_idx_i+y_idy_i)
\\ 
& \displaystyle
=2\e^{-2}(2r^2-1)z\beta-(2r^4-2r^2+1)
\left\{
2rdr+\e^{-2}
\left(
2zdz+2\sum_{i=1}^{n-1}(x_idx_i+y_idy_i)
\right)
\right\}.
\end{array}
$$
We see that the vector field $X$ also satisfies
$$
(dr-dz)(X)|\partial\Sigma=(r-1)^2\{\e^{-2}(-r^2+r+1)-(r+1)^2\}+\e(2r^2-1)
>0
$$
with attention to $z=r-1$. 
Thus $\partial \Sigma$ is contact-type. 
The singularity of $X|\Sigma$ is the union of 
$$
S_+(\Sigma)=\{((0,\theta,-\e\sqrt{1+\e}),0) \in U\times B_\e^{2n-2}\}
$$
and 
$$
S_-(\Sigma)=\{((0,\theta,+\e\sqrt{1+\e}),0) \in U\times B_\e^{2n-2}\}.
$$
They are respectively a source point and a sink point (see Figure 1).
\begin{figure}[h]
\centering
\includegraphics[height=55mm]{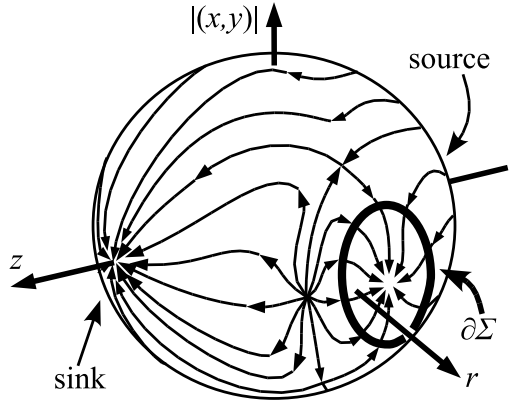}
\caption{$\F_\Sigma$ on $\Sigma (\approx S^{2n}\setminus N(S^1)\approx D^2\times S^{2n-2})$}
\end{figure}
Since the indices of these points are equal to $1$, 
the hypersurface $\Sigma$ violates the Thurston-Bennequin inequality.

Precisely, Figure 1 depicts (the fourfold covering of)
the well-defined push-forward $X'$ of $X$ under the natural projection $p$ from $\widetilde{\Sigma}$
to the quarter-sphere 
$$
\Sigma'=\left\{(z,r,|(x,y)|) \Mid r^2+\e^{-2}(z^2+|(x,y)|^2)=1+\e\right\} \quad (r\ge 0, \, |(x,y)|\ge 0).
$$
The vector field $X'$ defines the singular foliation
$\F'=\{\e^{-2}z^2=(Cr^2-1)(r^2-1)+\e\}_{-\infty \le C \le +\infty}$.
The singularity consists of the following five points; two (quarter-)elliptic points
$\left(\mp\e\sqrt{1+\e},0,0\right)$ whose preimages under $p$ are the above singular points;
other two (half-)elliptic points $\left(\pm\e\sqrt{\e},1,0\right)$ whose preimages are the periodic orbits 
$P_\pm=\{\pm\e\sqrt{\e}\}\times S^1(1)\times\{0\}\subset \widetilde{\Sigma} \subset \R \times \R^2 \times \R^{2n-2}$ 
of $X$ $(P_+\subset\Sigma, \, P_-\subset \widetilde{\Sigma}\setminus \Sigma)$;
and a hyperbolic point $\left(0,r_0=\sqrt{1+\e-\sqrt{\e(1+\e)}},n_0=\sqrt{\e^2\sqrt{\e(1+\e)}}\right)$, 
which is the self-intersection of the leaf corresponding to $C=1+2\e+2\sqrt{\e(1+\e)}$.
Slightly changing $\e$ if necessary, we may assume that the preimage 
$H=p^{-1}(\{(0,r_0,n_0)\})=\{0\}\times S^1(r_0) \times S^{2n-1}(n_0)$ of the hyperbolic 
singular point of $X'$
is a union of periodic orbits of $X$.

Now we assume that $\widetilde{\Sigma}$ is (approximately) convex
in order to prove Proposition by contradiction.
The assumption implies
$$
S_i(\Sigma), \, P_i \subset \widetilde{\Sigma}_i \quad (i=+,-)
\quad\textrm{and}\quad
\Gamma \cap H=\emptyset
$$
where $\widetilde{\Sigma}_\pm$ are the $\pm$ regions of 
the convex surface $\widetilde{\Sigma}$ devided by the transverse 
intersection with the level set $\Gamma$ of the contact Hamiltonian function
described in \S 1.2.
Then we can see that the intersection $\Gamma\cap\widetilde{\Sigma}$
contain a spherical component.
However the Eliashberg-Floer-McDuff theorem implies that 
$S^{2n-1}\coprod$(other components) can not be realized as the boundary of 
a connected convex symplectic manifold (see McDuff\cite{McDuff}). 
This contradiction proves Proposition. 
Here we omit a similar proof of the non-convexity of 
the Seifert hypersurface $\Sigma$.
\end{proof}
Theorem in \S1 is deduced from Proposition and the following easy lemma (see \cite{Mori2} for a proof).
\begin{lem}
There exists an embedded overtwisted contact $S^3$ topologically unknotted in $J^1(\R^2,\R)$.
\end{lem}

\end{document}